\documentclass[11pt]{amsart}

\usepackage[marginratio=1:1]{geometry}
\usepackage{color}
\definecolor{refs}{rgb}{0.7,0,0}
\definecolor{ext}{RGB}{112,112,112}
\definecolor{cite}{RGB}{034,113,179}
\usepackage[colorlinks=true,citecolor=cite,linkcolor=refs,urlcolor=ext,backref=page]{hyperref}

\newtheorem{theorem}{Theorem}
\newtheorem{lemma}{Lemma}

\newtheorem{proposition}{Proposition}

\theoremstyle{definition}
\newtheorem*{definition}{Definition}
\newtheorem*{remark}{Remark}

\newcommand{\R}{\mathbb{R}}

\newcommand{\const}{\mathrm{const}}
\newcommand{\spn}{\operatorname{span}}

\newcommand{\la}{\lambda}

\newcommand{\DD}{\mathcal{D}}

\newcommand{\FF}{\mathcal{F}}

\title[Two-component integrable extension of general heavenly equation]{Two-component integrable extension\\ of general heavenly equation}
\author[W.~Kry\'nski]{Wojciech Kry\'nski}
\address{
Institute of Mathematics, Polish Academy of Sciences, \'Sniadeckich 8, 00-656 Warsaw, Poland}

\email{krynski@impan.pl}

\author[A.~Sergyeyev]{Artur Sergyeyev}
\address{Mathematical Institute, Silesian University in Opava, Na Rybn\'i{}\v{c}ku 1, 746 01 Opava, Czech Republic}

\email{artur.sergyeyev@math.slu.cz}

\subjclass[2010]{}
\begin{document}

\begin{abstract}
We introduce an integrable two-component extension of the general heavenly equation and prove that the solutions of this extension are in one-to-one correspondence with 4-dimensional hyper-para-Hermitian metrics. Furthermore, we demonstrate that if the metrics in question are hyper-para-K\"ahler, then our system reduces to the general heavenly equation. We also present an infinite hierarchy of nonlocal symmetries and a recursion operator for the system under study.\looseness=-1 
\end{abstract}

\maketitle

\section{Introduction}\label{sec:intro}
The partial differential equation in four independent variables derived by Schief in \cite{Sch}, known as the general heavenly equation, serves as a natural framework for anti-self-dual gravity and has a number of important applications in physics and integrable systems \cite{DF, KSS, Kdef, MS, PS}. In the present paper, we introduce and study its two-component integrable extension, given below as system \eqref{system}. We demonstrate that solutions of this new system descend to a special class of metrics $[g]$ of neutral signature. These metrics admit a triple of anti-commuting vector bundle endomorphisms $I, J, K \colon TM \to TM$ such that $J$ is a complex structure, $I$ and $K$ are para-complex structures,
\[
IK=-KI=J
\]
and
\[
g(X,Y)=-g(JX,JY)=g(IX,IY)=g(KX,KY),
\]
for any $X,Y\in T_xM$, $x\in M$. Following \cite[Section 6.1]{IZ}, we shall refer to the structures as hyper-para-Hermitian (see also \cite{DGMY, DGM}). They are also known as pseudo-hyperhermitian in \cite{DJS, DW} or hyper-Hermitian of neutral signature in \cite{Kpleb}. The underlying algebra spanned by $(I, J, K)$ is called split-quaternions \cite{IZ} or para-quaternions \cite{DJS}, and it can be easily shown that the triple itself uniquely determines the conformal class $[g]$, see \cite{DGMY} and \cite{Kpleb} for higher-dimensional generalizations. The hyper-para-Hermitian structures are necessarily anti-self-dual and are characterized in the class of anti-self-dual metrics by the condition that the corresponding twistor space, defined as the space of $\alpha$-surfaces, fibers over a projective line \cite{DW}.

In this paper we work in a local coordinate system $(x^1,x^2,x^3,x^4)$ on manifold $M$. For any function $u\in C^\infty(M)$ we denote $u_i=\partial_{x^i}u$ and $u_{ij}=\partial_{x^i}\partial_{x^j}u$. Our main result is the following theorem giving a characterization of hyper-para-Hermitian metrics on 4-manifolds.
\begin{theorem}\label{thh}
Any hyper-para-Hermitian metric $[g]$ a 4-dimensional manifold can be locally put in the form 
\[
g=\omega^1\omega^4-\omega^2\omega^3
\]
where 1-forms $\omega^i$, $i=1,\ldots,4$ are defined as
\[
\begin{aligned}
&\omega^1=dx^4-\la_4\left(\frac{1}{\la_2-\la_4}\frac{v_1}{v_4}dx^1+\frac{1}{\la_1-\la_4}\frac{u_2}{u_4}dx^2\right)\\
&\omega^2=\frac{1}{\la_4}dx^4-\frac{1}{\la_4}\left(\frac{\la_2}{\la_2-\la_4}\frac{v_1}{v_4}dx^1+\frac{\la_1}{\la_1-\la_4}\frac{u_2}{u_4}dx^2\right)\\
&\omega^3=dx^3-\la_3\left(\frac{1}{\la_2-\la_3}\frac{v_1}{v_3}dx^1+\frac{1}{\la_1-\la_3}\frac{u_2}{u_3}dx^2\right)\\
&\omega^4=\frac{1}{\la_3}dx^3-\frac{1}{\la_3}\left(\frac{\la_2}{\la_2-\la_3}\frac{v_1}{v_3}dx^1+\frac{\la_1}{\la_1-\la_3}\frac{u_2}{u_3}dx^2\right)
\end{aligned}
\]
and functions $u$ and $v$ satisfy 
the following system
\begin{equation}\label{system}
\begin{aligned}
&(\la_3-\la_4)(\la_1-\la_2)v_1u_2u_{34}+(\la_2-\la_3)(\la_1-\la_4)(v_1u_4u_{23}-v_3u_4u_{12}+v_3u_2u_{14})\\
&\quad-(\la_2-\la_4)(\la_1-\la_3)(v_1u_3u_{24}-v_4u_3u_{12}+v_4u_2u_{13})=0,\\
&(\la_3-\la_4)(\la_1-\la_2)v_1u_2v_{34}+(\la_2-\la_3)(\la_1-\la_4)(v_1u_4v_{23}-v_3u_4v_{12}+v_3u_2v_{14})\\
&\quad-(\la_2-\la_4)(\la_1-\la_3)(v_1u_3v_{24}-v_4u_3v_{12}+v_4u_2v_{13})=0,
\end{aligned}
\end{equation}
where $\lambda_i\in\R$, $i=1,\ldots,4$, are fixed, arbitrary pairwise distinct constants. 

Moreover, the above system is integrable and admits an isospectral Lax pair with the Lax operators 
\[
\begin{aligned}
&L_0=(\la_2-\la_4)(\la_1-\la)u_2v_4\partial_{x^1}-(\la_1-\la_4)(\la_2-\la)u_4v_1\partial_{x^2}+(\la_1-\la_2)(\la_4-\la)u_2v_1\partial_{x^4}\\
&L_1=(\la_2-\la_3)(\la_1-\la)u_2v_3\partial_{x^1}-(\la_1-\la_3)(\la_2-\la)u_3v_1\partial_{x^2}+(\la_1-\la_2)(\la_3-\la)u_2v_1\partial_{x^3}
\end{aligned}
\]
\end{theorem}

Alternative descriptions of hyper-para-Hermitian structures through integrable systems have been studied in \cite{D, DW, Kpleb}. Notice that both equations in \eqref{system} share the same symbol, which, in turn, recovers the conformal class by means of the characteristic variety coinciding with the null cone of $[g]$. The relation of system \eqref{system} to the aforementioned general heavenly equation is explained in the following result.
\begin{theorem}\label{gen}
If $w$ satisfies the general heavenly equation 
\begin{equation}\label{ghe}
 (\lambda_2-\lambda_4)(\lambda_1-\lambda_3)w_{13}w_{24}- (\lambda_3-\lambda_4)(\lambda_1-\lambda_2)w_{12}w_{34}- (\lambda_2-\lambda_3)(\lambda_1-\lambda_4)w_{14}w_{23}=0
\end{equation}
then
\begin{equation}\label{potential}
u=w_1,\quad v=w_2
\end{equation}
satisfy \eqref{system}.

Moreover, a hyper-para-Hermitian structure is hyper-para-K\"ahler if and only if it can be locally put in the form from Theorem \ref{thh} with  $(u,v)$ given by \eqref{potential} for some function $w$ satisfying \eqref{ghe}.\looseness=-1
\end{theorem}

The hyper-para-K\"ahler structures appearing in Theorem \ref{gen} can be introduced as Ricci-flat hyper-para-Hermitian structures (see \cite{DW, MN}). According to the Mason--Newman formalism \cite{MN}, they are described by an integrable system with a Lax pair whose Lax operators are vector fields that are divergence-free with respect to some volume form. One realization of this formalism is provided by the general heavenly equation \eqref{ghe} (we refer to \cite{KSS,MS,Sch} for details). Other approaches include the Pleba\'nski equations and the Husain--Park equation, among others, cf.\ \cite{KSS,PS}.

The rest of the paper is organized as follows. 

Section \ref{sec:hermitian} gives the proofs of Theorems \ref{thh} and \ref{gen}. We apply the approach from \cite{Kpleb}, where the hyper-para-Hermitian structures are studied through the lens of Kronecker webs. Recall that the webs are special families of foliations, originally introduced as reductions of certain bi-Hamiltonian systems \cite{Z}. Similar concepts have also been recently exploited in \cite{PS} and earlier in \cite{DK,Khirota} in the context of 3-dimensional Einstein--Weyl geometry.

Section \ref{sec:symmetries} explores the properties of \eqref{system} as an integrable system. In particular,  we present an infinite hierarchy of nonlocal symmetries and a recursion operator for \eqref{system}. 

\section{Hyper-para-hermitian structures}\label{sec:hermitian}
The aim of this section is to give proofs of Theorems \ref{thh} and \ref{gen}. As mentioned in the Introduction, we shall use a description of hyper-para-Hermitian structures via Kronecker webs, which was recently exploited in \cite{Kpleb} and \cite{PS}. In the 4-dimensional case studied in this paper, the webs are defined as follows.
\begin{definition}
A \emph{Kronecker web} on a 4-dimensional manifold is a 1-parameter family of foliations $\{\FF_\la\}_{\la\in \R}$ such that, locally, there exist point-wise independent 1-forms $\alpha^1,\ldots,\alpha^4$ such that
\[
T\FF_\la=\ker\left(\alpha^1+\la\alpha^2,\alpha^3+\la\alpha^4\right).
\]
\end{definition}

Notice that we can equivalently write
\[
T\FF_\la=\spn\left\{X_0(\la), X_1(\la) \right\}
\]
for $X_0(\la)=Y_2-\la Y_1$, $X_1(\la)=Y_4-\la Y_3$, where $(Y_i)_{i=1,\ldots,4}$ is a frame on $M$ dual to the co-frame $(\alpha^i)_{i=1,\ldots,4}$.

In \cite{Kpleb}, a one-to-one correspondence between hyper-para-Hermitian structures and Kronecker webs on 4-dimensional manifolds is established. Here, the endomorphisms $I$ and $K$ are defined such that the corresponding eigenspaces $\DD_I^+$, $\DD_I^-$, $\DD_K^+$ and $\DD_K^-$, associated with the eigenvalues $\pm 1$, respectively, are given by $T\FF_{\lambda_i}$ for certain fixed values of $\lambda_i$, $i=1,2,3,4$ (see Corollaries 2.3, 2.5 and the discussion on page 461 in \cite{Kpleb}). Note that the complex structure $J$ is induced from $I$ and $K$ by means of the composition $J=IK$. Additionally, from the viewpoint of the anti-self-dual metrics, the leaves of foliations $\{\FF_\la\}_{\la\in \R}$ are $\alpha$-submanifolds of $[g]$ and the collection of all leaves is usually referred to as the twistor space, see \cite{DW}. 
\looseness=-1 
\vskip 2ex
\noindent
{\em Proof of Theorem \ref{thh}.}
Let $\la_1,\ldots,\la_4\in\R$ be four distinct values of $\lambda$. There are functions $f^i,g^i$, $i=1,\ldots,4$, defined locally in a neighborhood of a point $x\in M$, such that the corresponding foliations $\FF_{\la_i}$ are defined as $f^i=\const$, $g^i=\const$, respectively. Without lost of generality, since the foliations $\FF_\la$ corresponding to different values of $\la$ are transversal, one can define local coordinates in the neighborhood of $x$ as $x^i=f^i$. Moreover, let $u=g^1$ and $v=g^2$.

We shall look for the $\lambda$-dependent vector fields $X_0(\lambda)$ and $X_1(\lambda)$, spanning $T\FF_\la$ for any $\la\in\R$. They can be written in the most general form as
\[
X_j=(a^1_j+\la b^1_j)\partial_{x^1}+(a^2_j+\la b^2_j)\partial_{x^2}+(a^3_j+\la b^3_j)\partial_{x^3}+(a^4_j+\la b^4_j)\partial_{x^4}, \qquad j=0,1,
\]
for some functions $a^i_j$ and $b^i_j$, $i=1,\ldots,4$, $j=0,1$, in a neighborhood of $x$. Now, since $dx^i$ annihilates both $X_0(\la_i)$ and $X_1(\la_i)$ we get $a^i_j=-\la_ib^i_j$ and consequently
\[
X_j=(\la-\la_1)b^1_j\partial_{x^1}+(\la-\la_2)b^2_j\partial_{x^2}+(\la-\la_3)b^3_j\partial_{x^3}+(\la-\la_4) b^4_j\partial_{x^4}, \qquad j=0,1.
\]
Further on, $du$ annihilates $X_0(\la_1)$ and $X_1(\la_1)$ and $dv$ annihilates $X_0(\la_2)$ and $X_1(\la_2)$. These conditions give equations
\begin{equation}\label{eq:b}
\begin{aligned}
(\la_1-\la_2)u_2b^2_j+(\la_1-\la_3)u_3b^3_j+(\la_1-\la_4)u_4b^4_j=0,&\qquad j=0,1\\
(\la_2-\la_1)v_1b^1_j+(\la_2-\la_3)v_3b^3_j+(\la_2-\la_4)v_4b^4_j=0,&\qquad j=0,1.
\end{aligned}
\end{equation}
Moreover, $X_1$ and $X_2$ are defined up to a $GL(2)$-action. Hence, it can be assumed that $b^3_0=0$, $b^4_0=1$, $b^3_1=1$ and $b^4_1=0$. There are four $b^i_j$ left, with $i=1,2$ and $j=0,1$. However, system \eqref{eq:b} fixes them uniquely:
\[
b^1_0=\frac{(\la_2-\la_4)}{(\la_1-\la_2)}\frac{v_4}{v_1},\quad b^1_1=\frac{(\la_2-\la_3)}{(\la_1-\la_2)}\frac{v_3}{v_1},\quad b^2_0=-\frac{(\la_1-\la_4)}{(\la_1-\la_2)}\frac{u_4}{u_2},\quad b^2_1=-\frac{(\la_1-\la_3)}{(\la_1-\la_2)}\frac{u_3}{u_2} . 
\]
We get
\begin{equation}\label{x}
\begin{aligned}
&X_0=(\la-\la_1)\frac{(\la_2-\la_4)}{(\la_1-\la_2)}\frac{v_4}{v_1}\partial_{x^1}-(\la-\la_2)\frac{(\la_1-\la_4)}{(\la_1-\la_2)}\frac{u_4}{u_2}\partial_{x^2}+(\la-\la_4)\partial_{x^4},\\
&X_1=(\la-\la_1)\frac{(\la_2-\la_3)}{(\la_1-\la_2)}\frac{v_3}{v_1}\partial_{x^1}-(\la-\la_2)\frac{(\la_1-\la_3)}{(\la_1-\la_2)}\frac{u_3}{u_2}\partial_{x^2}+(\la-\la_3)\partial_{x^3}.
\end{aligned}
\end{equation}
Note that setting $L_i=(\la_1-\la_2)u_2v_1X_i$ recovers the Lax operators from Theorem~\ref{thh}. 
By assumption, $X_0(\la)$ and $X_1(\la)$ span an integrable distribution for any $\la$. Hence,  $X_0(\la)$ and $X_1(\la)$ necessarily commute. One readily checks that vanishing of the commutator of $X_i(\la)$ is equivalent to system \eqref{system}.  Likewise, it is easily seen that vanishing of the commutator of $L_i$ modulo a certain linear combination of $L_i$ is equivalent to system \eqref{system}. Thus an overdetermined linear system for $\psi$, $L_i\psi=0$, $i=0,1$, whose compatibility condition is nothing but \eqref{system}, provides a Lax pair for \eqref{system}. 
$\hfill\Box$\looseness=-1

\begin{lemma}\label{lem}
Suppose that $(u,v)$ is a solution to \eqref{system}. Then, for any two nowhere vanishing functions $a,b\colon\R^2\to\R$, the pair $(\tilde u,\tilde v)$ defined as
\[
\begin{aligned}
&\tilde u(x^1,x^2,x^3,x^4)=a(x^1,u(x^1,x^2,x^3,x^4))u(x^1,x^2,x^3,x^4),\\
&\tilde v(x^1,x^2,x^3,x^4)=b(x^2,v(x^1,x^2,x^3,x^4))v(x^1,x^2,x^3,x^4),\\
\end{aligned}
\]
is a solution to \eqref{system} descending to the same hyper-para-Hermitian metric.
\end{lemma}
\begin{proof}
Observe that, in view of the proof of Theorem \ref{thh}, the functions $\tilde u$ and $\tilde v$ belong to the rings of functions constant on leaves of foliations $\FF_{\la_1}$ and $\FF_{\la_2}$, respectively.
Hence $\tilde u$ and $\tilde v$ can replace $u$ and $v$ in the derivation of the Lax pair in the proof of Theorem \ref{thh}, and therefore in system \eqref{system} itself, and the result follows.
\end{proof}
\vskip 1ex
\noindent
{\em Proof of Theorem \ref{gen}.}
The first part of the theorem can be verified by direct computations. In order to prove the second part, we shall work using the coordinates from Theorem \ref{thh}. The Lax distribution $\spn\{L_0(\lambda),L_1(\lambda)\}$ is annihilated by two $\lambda$-dependent 1-forms that can be taken in the following form
\[
\begin{aligned}
\alpha_0(\lambda)=&(\la_1-\la_2)(\la_4-\la)\left(\frac{u_3}{u_2}dx^1+\frac{v_3}{v_1}dx^2+\frac{u_3v_3}{u_2v_1}dx^3\right)\\
&+\left((\la_1-\la_4)(\la_2-\la)\frac{u_4v_3}{u_2v_1}-(\la_2-\la_4)(\la_1-\la)\frac{u_3v_4}{u_2v_1}\right)dx^4\\
\alpha_1(\lambda)=&(\la_1-\la_2)(\la_3-\la)\left(\frac{u_4}{u_2}dx^1+\frac{v_4}{v_1}dx^2+\frac{u_4v_4}{u_2v_1}dx^3\right)\\
&+\left((\la_2-\la_3)(\la_1-\la)\frac{u_4v_3}{u_2v_1}-(\la_1-\la_3)(\la_2-\la)\frac{u_3v_4}{u_2v_1}\right)dx^4\\
\end{aligned}
\]
Let $\beta=\alpha_0\wedge\alpha_1$. According to \cite[Corollary 2.4]{PS} a hyper-para-Hermitian structure is divergence-free (which in the terminology of the present paper means that the structure is hyper-para-K\"ahler) if and only if there is a nowhere vanishing function $c\colon M\to \R$ such that $c\beta$ is closed. Since $L_0$ and $L_1$ commute modulo a linear combination thereof with nonconstant coefficients, see the proof of Theorem~\ref{thh}, we have
\[
d\alpha_i=0\mod \alpha_0,\alpha_1,
\]
and consequently
\[
d\beta=\varphi\wedge\beta
\]
where $\varphi$ is explicitly written as
\[
\begin{aligned}
\varphi(\lambda)=\frac{1}{(\la_1-\la_2)(\la-\la_3)u_2v_1}&\left((\la_2-\la_3)(\la-\la_1)(u_{12}(v_3-u_{23}v_1)\right.\\
&\left.+(\la_1-\la_3)(\la-\la_2)(u_2v_{13}-u_3v_{12})\right)dx^3\\
+\frac{1}{(\la_1-\la_2)(\la-\la_4)u_2v_1}&\left((\la_2-\la_4)(\la-\la_1)(u_{12}v_4-u_{24}v_1)\right.\\
&\left.+(\la_1-\la_4)(\la-\la_2)(u_2v_{14}-u_4v_{12})\right)dx^4\\
\end{aligned}
\]
Notice that $\varphi$ is specified only up to transformations $\varphi\mapsto \varphi_{A^0,A^1}=\varphi+A^0\alpha_0+A^1\alpha_1$ for some functions $A^0$ and $A^1$. Since the result is local, we can apply the Poincar\'e lemma for differential forms, and get that function $c$ exists if and only if $\varphi_{A^0,A^1}$ is exact for an appropriate choice of $A^0$ and $A^1$. Moreover $\varphi_{A^0,A^1}$ has to be a function on $M$ (i.e. it cannot depend on $\lambda$). We now consider the equation
\[
\varphi_{A^0,A^1}=\varphi+A^0\alpha_0+A^1\alpha_1=df
\]
where $A^0,A^1$ and $f$ are unknown. Examining the coefficients at $dx^1$ and $dx^2$, we obtain an algebraic system for $A^i$'s. This system can be effortlessly solved, resulting in
\[
A^0=\frac{1}{(\la_1-\la_2)(\la-\la_4)}\left(\frac{u_4v_1f_2-u_2v_4f_1}{u_3v_4-u_4v_3}\right),\ \  A^1=\frac{1}{(\la_1-\la_2)(\la-\la_3)}\left(\frac{u_2v_3f_1-u_3v_1f_2}{u_3v_4-u_4v_3}\right).
\]
Then, the coefficients at $dx^3$ and $dx^4$ reduce to differential equations for $f$, which, in turn, can be written in the form
\begin{equation}\label{eqforf}
L_0(f)=(\la-\la_4)\varphi_4,\qquad L_1(f)=(\la-\la_3)\varphi_3,
\end{equation}
where $\varphi_3$ and $\varphi_4$ stand for the coefficients defined as $\varphi=\varphi_3dx^3+\varphi_4dx^4$. The system always satisfies the compatibility condition, which can be verified directly. Hence, it yields a solution $f$. However, in general, $f$ depends on $\lambda$. Imposing the condition that $\partial_\la f=0$ and substituting $\la=\la_1$ and $\la=\la_2$ in \eqref{eqforf} gives the following system
\[
\begin{aligned}
&u_2v_1f_3-u_3v_1f_2=u_2v_{13}-u_3v_{12},\\
&u_2v_1f_3-u_2v_3f_1=u_{23}v_1-u_{12}v_3,\\
&u_2v_1f_4-u_4v_1f_2=u_2v_{14}-u_4v_{12},\\
&u_2v_1f_4-u_2v_4f_1=u_{24}v_1-u_{12}v_4.\\
\end{aligned}
\]
It turns out that this system can be rewritten as
\[
\begin{aligned}
&\left(\partial_{x^3}-\frac{u_3}{u_2}\partial_{x^2}\right)(f-\ln(v_1))=0,\qquad \left(\partial_{x^3}-\frac{v_3}{v_1}\partial_{x^1}\right)(f-\ln(u_2))=0,\\
&\left(\partial_{x^4}-\frac{u_4}{u_2}\partial_{x^2}\right)(f-\ln(v_1))=0,\qquad \left(\partial_{x^4}-\frac{v_4}{v_1}\partial_{x^1}\right)(f-\ln(u_2))=0.\\
\end{aligned}
\]
Denoting $g=f-\ln(v_1)$ and $h=f-\ln(u_2)$ and using the fact that the vector fields $\partial_{x^3}-\frac{u_3}{u_2}\partial_{x^2}$ and $\partial_{x^4}-\frac{u_4}{u_2}\partial_{x^2}$ commute, and so do $\partial_{x^3}-\frac{v_3}{v_1}\partial_{x^1}$ and $\partial_{x^4}-\frac{v_4}{v_1}\partial_{x^1}$, we get that $g$ and $h$ can be written as $g=g(x^1,u)$ and $h=h(x^2,v)$, because $x^1$ and $u$ are constant along the first pair of vector fields, and $x^2$ and $v$ are constant along the second pair of vector fields. Further, expressing $f$ in terms of $g$ and $h$ and comparing the expressions gives
\[
\ln(u_2)-\ln(v_1)=h-g.
\]
By Lemma \ref{lem} there is a different solution to \eqref{system} such that $\ln(\tilde u_2)-\ln(\tilde v_1)=0$. Indeed, one can adjust functions $a$ and $b$ from Lemma \ref{lem}, such that $h$ and $g$ cancel out. Consequently,
\[
\ln\left(\frac{\tilde u_2}{\tilde v_1}\right)=0.
\]
Hence, we have $\tilde u_2=\tilde v_1$ and it implies that the 1-form
\[
\tilde u dx^1+\tilde v dx^2
\]
is closed with respect to coordinates $(x^1,x^2)$, meaning that locally $\tilde u=w_1$ and $\tilde v =w_2$, for some function $w$, i.e. \eqref{potential} holds. Furthermore, the Lax pair in Theorem \ref{thh}, under assumption \eqref{potential}, reduces to the Lax pair for \eqref{ghe}, as can be found, for instance, in \cite[formula (4.11)]{KSS}. It follows that $w$ is a solution to \eqref{ghe}.$\hfill\Box$

\begin{remark}
The second assertion of Theorem \ref{gen} can be expressed in terms of a curvature. Namely, in order for a hyper-para-Hermitian structure to be hyper-para-K\"ahler the Obata connection associated with the former 
has to be Ricci-flat. Recall that the Obata connection is the unique connection $\nabla$ such that $\nabla I=\nabla J=\nabla K=0$. It was initially introduced within the framework of hyper-Hermitian structures \cite{Obata}, but its definition extends seamlessly to the neutral signature, see \cite{AMT}. In this context it coincides with the Chern connection of webs \cite{N}, defined by the following expression (see also \cite[formula (2.1)]{Kpleb})\looseness=-1
\[
\nabla_XY=\pi_H(j[\pi_H(X),j\pi_H(Y)]+[\pi_V(X),\pi_H(Y)])+\pi_V(j[\pi_V(X),j\pi_V(Y)]+[\pi_H(X),\pi_V(Y)])
\]
where $[.,.]$ is the Lie bracket of vector fields, and $\pi_V$ and $\pi_H$ are projections to the factors of the decomposition $TM=V\oplus H$, where
\[
V=\spn\{L_0(\lambda_1),L_1(\lambda_1)\},\quad H=\spn\{L_0(\lambda_2),L_1(\lambda_2)\},
\]
and $j\colon TM\to TM$ is a mapping satisfying $j^2=1$ and uniquely determined by properties
\[
j\colon V\to H,\quad j\colon H\to V,\quad j(\pi_H(X))-\pi_H(X)\in T, \quad j(\pi_V(X))-\pi_V(X)\in T,
\]
for any $X\in TM$, where $T=\spn\{L_0(\lambda_3),L_1(\lambda_3)\}$. In the present context
\[
V=\spn\{u_4\partial_2-u_2\partial_4,u_3\partial_2-u_2\partial_3\},\quad H=\spn\{v_4\partial_1-v_1\partial_4,v_3\partial_1-v_1\partial_3\}
\]
and
\[
\begin{aligned}
&j(u_3\partial_2-u_2\partial_3)=\frac{u_2}{v_1}(v_3\partial_1-v_1\partial_3),\quad j(u_4\partial_2-u_2\partial_4)=C\frac{u_2}{v_1}(v_4\partial_1-v_1\partial_4),\\
&j(v_3\partial_1-v_1\partial_3)=\frac{v_1}{u_2}(u_3\partial_2-u_2\partial_3),\quad 
j(v_4\partial_1-v_1\partial_4)=C^{-1}\frac{v_1}{u_2}(u_4\partial_2-u_2\partial_4).
\end{aligned}
\]
where $C=\frac{(\lambda_3-\lambda_1)(\lambda_4-\lambda_2)}{(\lambda_4-\lambda_1)(\lambda_3-\lambda_2)}$.
\end{remark}

\section{Symmetries and Recursion Operator}\label{sec:symmetries}
Now turn to the study of symmetries for \eqref{system}. To this end we first note that the operators $X_i$ given by \eqref{x} from the proof of Theorem \ref{thh} are linear in $\la$ and thus can be written as
\begin{equation}\label{lax-exp}
X_i=X_i^{(0)}-\lambda X_i^{(1)},\quad i=0,1 
\end{equation}
with obvious expressions for $X_i^{(j)}$.

Consider now the following `adjoint Lax pair' for \eqref{system}:
\begin{equation}\label{alp}
[Q, X_i]=0, \quad i=0,1, 
\end{equation}
where $[\cdot,\cdot]$ is again the usual Lie bracket of vector fields 
and
\begin{equation}\label{q}
Q=\sum\limits_{j=1}^2 
\xi^j\partial_{x^j} 
\end{equation}
(in this connection recall that $[X_0,X_1]=0$ modulo \eqref{system}). 

The formal expansions
\begin{equation}\label{expxi}
\xi^j=\sum\limits_{r=0}^\infty \xi_r^j\lambda^r,\quad j=1,2 
\end{equation}
give rise to an infinite hierarchy of nonlocal variables $\xi_r^j$ associated with \eqref{system} as follows.

Let $Q_r=\sum\limits_{j=1}^2 \xi_r^j\partial_{x^j}$.  Then substituting \eqref{q} and \eqref{expxi} into \eqref{alp} yields 
\begin{equation}\label{alp-0}
[Q_{0},X_i^{(0)}]=0,\quad i=0,1
\end{equation}
and a set of recursion relations
\begin{equation}\label{alp-exp}
[Q_r,X_i^{(0)}]=[Q_{r-1},X_i^{(1)}],\quad i=0,1,\quad r=1,2,\dots
\end{equation} 
relating $Q_r$ and $Q_{r-1}$. 

Equations \eqref{alp-0} and \eqref{alp-exp} can be solved with respect to $\partial \xi_r^j/\partial x^m$, $m=3,4$, to yield relations of general form  
\[
\partial \xi_r^j/\partial x^m=A_{mr}^j, \quad m=3,4,\quad j=1,2,\quad r=0,1,2,\dots
\]
which recursively define $\xi_r^j$ starting from $r=0$; here $A_{mr}^j$ are certain functions, a bit too cumbersome to spell out here in full, of $\xi_s^j$, $s=0,\dots,r$ and $x$- and $y$-derivatives of those, and of a number of first- and second-order derivatives of $u$ and $v$.  

In other words, we have here an infinite-dimensional differential covering over \eqref{system} with the nonlocal variables $\xi_r^j$, $j=1,2$, $r=0,1,2,\dots$; the said covering is defined via \eqref{alp-0} and \eqref{alp-exp}.
For generalities on nonlocal variables, differential coverings and nonlocal symmetries the reader is referred to \cite{KVV,MS0,AS} and references therein.\footnote{Note, however, that the authors of \cite{KVV} refer to what we call below nonlocal symmetries as to the shadows of nonlocal symmetries.}

With this in mind, we readily arrive at the following result 
\begin{proposition}\label{sym}
\label{flow-gf}
The flows 
\begin{equation}\label{fls}
u_{\tau_r}=\xi_r^2 u_2,\quad v_{\tau_r}=-\xi_r^1 v_1,\quad r=0,1,2,\dots,
\end{equation}
with $\xi_r^k$ defined above, are compatible with \eqref{system}, and thus define an infinite hierarchy of  
nonlocal symmetries for \eqref{system} with the characteristics $(\xi_r^2 u_2,-\xi_r^1 v_1)$, $r=0,1,2,\dots$.
\end{proposition} 

Note that such infinite hierarchies of symmetries are a common occurrence for integrable partial differential systems, and for systems in more than two independent variables the symmetries in question are usually nonlocal, cf.\ e.g.\ \cite{KVV, MS0, MS1,O,AS} and references therein.


\noindent
{\em Sketch of proof.} It is straightforward to verify that $U=\xi^2 u_2$ and $V=-\xi^1 v_1$ satisfy the linearized version of \eqref{system}, that is,
\begin{equation}\label{sys-lin}
\begin{aligned}
&(\la_3-\la_4)(\la_1-\la_2)(V_1 u_2u_{34}+v_1 U_2u_{34}+v_1u_2 U_{34})\\
&\quad+(\la_2-\la_3)(\la_1-\la_4)(V_1u_4u_{23}+v_1U_4u_{23}+v_1u_4U_{23}\\
&\quad-V_3u_4u_{12}-v_3U_4u_{12}-v_3u_4 U_{12}+V_3u_2u_{14}+v_3U_2u_{14}+v_3 u_2U_{14})\\
&\quad-(\la_2-\la_4)(\la_1-\la_3)(V_1u_3u_{24}+v_1U_3u_{24}+v_1u_3U_{24}\\
&\quad-V_4u_3u_{12}-v_4U_3u_{12}-v_4u_3U_{12}+V_4u_2u_{13}+v_4U_2u_{13}+v_4u_2U_{13})=0,\\
&(\la_3-\la_4)(\la_1-\la_2)(V_1u_2v_{34}+v_1U_2v_{34}+v_1u_2V_{34})\\
&\quad+(\la_2-\la_3)(\la_1-\la_4)(V_1u_4v_{23}+v_1U_4v_{23}+v_1u_4V_{23}\\
&\quad-V_3u_4v_{12}-v_3U_4v_{12}-v_3u_4V_{12}+V_3u_2v_{14}+v_3U_2v_{14}+v_3u_2V_{14})\\
&\quad-(\la_2-\la_4)(\la_1-\la_3)(V_1u_3v_{24}+v_1U_3v_{24}+v_1u_3V_{24}\\
&\quad-V_4u_3v_{12}-v_4U_3u_{12}-v_4u_3V_{12}+V_4u_2v_{13}+v_4U_2v_{13}+v_4u_2V_{13})=0,
\end{aligned}
\end{equation}
modulo \eqref{system}, \eqref{alp} and differential consequences thereof, i.e., \eqref{system} admits a nonlocal symmetry with the characteristic $(\xi^2 u_2, -\xi^1 v_1)$ . Upon taking into account linearity of \eqref{sys-lin} in $U$ and $V$ and substituting the expansions \eqref{expxi} into \eqref{sys-lin} it is readily seen that 
$U=\xi_s^2 u_2$ and $V=-\xi_s^1 v_1$ for $s=0,1,2,\dots$ also satisfy \eqref{sys-lin} modulo \eqref{system}, \eqref{alp-0}, and \eqref{alp-exp} for $r=1,\dots,s$, and differential consequences thereof, so \eqref{system} indeed admits an infinite hierarchy of  nonlocal symmetries with the characteristics $(\xi_s^2 u_2,-\xi_s^1 v_1)$, $s=0,1,2,\dots$. Compatibility of \eqref{fls} with \eqref{system} immediately follows from this,
cf.\ e.g.\ \cite{KVV, O} and references therein. $\hfill\Box$\looseness=-1


In closing note that, using the technique from \cite{AS}, we have arrived at the following result, which is readily verified by straightforward albeit fairly tedious computation (for background on total derivatives and recursion operators seen essentially as B\"acklund auto-transformations for the linearized version of the system under study see e.g.\ \cite{KVV,MS0,MS1,AS}):\looseness=-1 

\begin{proposition}\label{recop}
Let the flow 
\begin{equation}\label{flwUV}
u_{\tau}=U,\quad v_{\tau}=V, 
\end{equation}
where $U$ and $V$ are assumed to be functions of independent variables, jet variables intrinsic to \eqref{system}, and some nonlocal variables over \eqref{system},  be compatible with \eqref{system} and thus define a nonlocal symmetry with the characteristic $(U,V)$ for \eqref{system}.

\looseness=-1

Consider nonlocal variables $\zeta^i$ defined by the relations
\begin{equation}\label{r-o}
[X_i^{(1)},S]=\left([X_i^{(0)},R]\right)^{TD},\quad i=0,1
\end{equation}
where $X_i^{(j)}$ are defined via \eqref{lax-exp}, 
and 
\begin{equation}\label{R-S}
R=-(V/v_1)\partial_{x^1}+(U/u_2)\partial_{x^2},\quad S=\sum\limits_{j=1}^2 \zeta^j\partial_{x^j}
\end{equation}
and the superscript $TD$ means that expressions involving the derivatives of $U$ and $V$ like $U_{x^i}$ and $V_{x^j}$  should be replaced by $D_{x^i}U$ and $D_{x^j}V$, where $D_{x^i}$ are total derivatives.

Then the new flow
\begin{equation}\label{new-flw}
u_{\sigma}=\tilde{U},\quad v_{\sigma}=\tilde{V}, 
\end{equation} 
where
\begin{equation}\label{recUV}
\tilde U=u_2\zeta^2,\quad \tilde{V}=-v_1\zeta^1,
\end{equation}
is also compatible  with \eqref{system} and thus defines another nonlocal symmetry with the characteristic $(\tilde{U},\tilde{V})$ for \eqref{system}.

In other words, the relations \eqref{r-o}, \eqref{R-S}, and \eqref{recUV} define a recursion operator for \eqref{system} associating to any nonlocal symmetry with the characteristic $(U,V)$ for \eqref{system} a new (again in general nonlocal) symmetry with the characteristic $(\tilde{U},\tilde{V})$ 
for \eqref{system}.
\end{proposition}

The above 
is a natural generalization of the recursion operator for general heavenly equation found in \cite{AS}. 

Notice that one now can construct (in general nonlocal) symmetries for system \eqref{system} beyond those from Proposition~\ref{sym} by applying the recursion operator from Proposition~\ref{recop} e.g.\ to `simpler' symmetries like, say, those with the characteristics $(u_i,v_i)$, $i=1,\dots,4$, that result from the translation invariance of \eqref{system}.%

\subsection*{Acknowledgments}
The research of WK, and the visit of AS to Warsaw, were supported in part by the grant 2019/34/E/ST1/00188 from the National Science Centre (NCN), Poland.  The research of AS was also supported in part through institutional funding for the development of research organizations (RVO) for I\v{C} 47813059. The authors gratefully acknowledge the support from the above sources.

WK and AS thank respectively Silesian University in Opava and Institute of Mathematics of Polish Academy of Sciences in Warsaw for warm hospitality extended to them in the course of their visits to the respective institutions. 

Some of the computations in the paper were performed  employing the package {\em Jets} \cite{BM} for Maple\textsuperscript{\textregistered} whose use is hereby gratefully acknowledged.

\end{document}